\documentclass[11pt]{article}
\usepackage{amssymb}

\usepackage[svgnames]{xcolor}
\usepackage{tikz}
\pgfdeclarelayer{edgelayer}
\pgfdeclarelayer{nodelayer}
\pgfsetlayers{edgelayer,nodelayer,main}

\tikzstyle{none}=[inner sep=0pt]
\tikzstyle{edge}=[draw=black]
\tikzstyle{vertex}=[circle, scale=.5, fill=Black,draw=Black]
\tikzstyle{textbox}=[rectangle,fill=none,draw=none]

\usepackage[left=3cm,top=3cm,right=3cm,bottom=3cm]{geometry}
\newtheorem{theorem}{Theorem}[section]
\newtheorem{lemma}[theorem]{Lemma}

\newtheorem{proposition}[theorem]{Proposition}

\newtheorem{corollary}[theorem]{Corollary}

\newcommand{\noin}{\noindent}
\newcommand{\qed}{\ \hfill \rule{1ex}{1ex}} %rightjustified box

\newenvironment{proof}{{\noin \bf Proof}: }{\qed \\}

\begin{document}

\title{Hamiltonicity of Bell and Stirling colour graphs}
    \author{S.~Finbow \thanks{Research supported by NSERC.
{\tt sfinbow@stfx.ca}, {\tt gmacgill@uvic.ca}}, $\mbox{G.~MacGillivray}^*$}
    \date{ }
    \maketitle

\begin{abstract}
For a graph $G$ and a positive integer $k$, the $k$-Bell colour graph of $G$
is the graph whose vertices are the partitions of $V$ into at most $k$ independent
sets, with two of these being adjacent if there exists a vertex $x$ such that 
the partitions are identical when restricted to $V - \{x\}$.
The $k$-Stirling Colour graph of $G$ is defined similarly, but for partitions into
exactly $k$ independent sets.

We show that every graph on $n$ vertices, except $K_n$ and $K_n - e$, has a Hamiltonian 
$n$-Bell colour graph, and this result is best possible.
It is also shown that, for $k \geq 4$, the $k$-Stirling colour graph of a tree with at 
least $k+1$ vertices is Hamiltonian, and the 3-Bell colour graph of a tree with at least
3 vertices is Hamiltonian.
\end{abstract}

\bigskip


\medskip


\section{ Introduction}

For a set $S$ of combinatorial objects, the \emph{reconfiguration graph for $S$}
has the elements of $S$ as its vertices, with two of these being adjacent 
if they differ in some small, specified way.
When this graph is connected, any element of $S$ can be reconfigured into any other
via a sequence of small changes of the specified type.
When it is Hamiltonian, there is a cyclic list that contains all elements of $S$,
and consecutive elements of the list differ by a small change of the 
specified type (a cyclic \emph{combinatorial Gray code} for the objects in question).
The paper by Ito \cite{Ito} and the survey by van den Heuvel \cite{vdH} give a sense 
of the literature and wide variety of reconfiguration problems that have been considered.
The survey by Savage \cite{Savage} gives pointers to the vast literature on combinatorial Gray codes.
In this paper, we are  interested in combinatorial Gray codes for a type of graph colouring.

A considerable amount is known about the \emph{$k$-colour graph} of a graph $G$,  denoted by $\mathcal{C}_k(G)$ (e.g., see \cite{vdH}).
It has the $k$-colourings of $G$ enumerated by the chromatic polynomial as vertices, 
with two of these being adjacent when they differ in the colour of exactly one vertex.  
This graph is
 connected whenever $k$ is at least one more than the colouring number of $G$ \cite{DFFV},
and is Hamiltonian when $k$ is at least two more than the colouring number of $G$ \cite{CM}.
Necessary, and usually sufficient conditions, on $k$ for the existence of a Hamilton cycle in a $k$-colour graph of  a tree, cycle, complete graph, complete bipartite graph, 2-tree, or complete multipartite graph have been found \cite{Bard, CS, CCMS, CM}.

Much less is known about the {\em canonical $k$-colouring of a graph $G$ with respect to a vertex ordering $\pi = x_1, x_2, \ldots, x_n$}.
A $k$-colouring of $G$ is  \emph{canonical} if it has the property that, for every $c \geq 1$, if $x_j$ is assigned colour $c$ then each of the colours $1, 2, \ldots, c-1$ has been assigned some vertex that precedes $x_j$ in $\pi$.  The {\em canonical $k$-colour graph of $G$ with respect to $\pi$}, denoted $\mathrm{Can}_k^\pi(G)$, has as vertices the canonical $k$-colourings of $G$ with respect to $\pi$, with two canonical $k$-colourings $c_1$ and $c_2$ being adjacent when they agree on all but one vertex of $G$.   Haas proved that for any tree $T$ with at least 4 vertices there exists a vertex ordering $\pi$ such that
$\mathrm{Can}_k^\pi(G)$ has a Hamilton cycle for all $k \geq 3$ \cite{Haas}.  By contrast, a  canonical $k$-colouring graph of a complete multipartite graph almost never has a Hamilton path or cycle, but   for all $k \geq 3$ there exists a vertex ordering $\pi$ such that $\mathrm{Can}_k^\pi(K_{m,n})$ has a Hamilton path for $m,n \geq 2$ \cite{HM}.

%
%
%At a talk by G. MacGillivray at the 29th ACCMCC and 2004 NZIMA Meeting in Lake Taupo, New Zealand, D. Archdeacon proposed studying Hamiltonicity of the graph colourings in the case where a $k$-colouring of a graph $G$ is regarded as a partition of the vertex set of $G$ into $k$ cells, each of which is a non-empty independent set.   

For a given graph $G$ and positive integer $k$, the {\em $k$-Bell colour graph of $G$}, denoted $\mathcal{B}_k(G)$, has as vertices the set of partitions of $V(G)$ into $k$ or fewer independent sets, with different partitions $p_1$ and $p_2$ being adjacent if there is a vertex $x \in V(G)$ such that the restrictions of $p_1$ and $p_2$ to $V(G)-\{x\}$ are equal. 
The  {\em $k$-Stirling colour graph}, $\mathcal{S}_k(G)$,  is defined similarly, but for partitions of $V(G)$ into exactly $k$ independent sets.  Observe that
$\mathcal{S}_k(G)$ is the  subgraph of $\mathcal{B}_k(G)$  induced by $V(\mathcal{B}_k(G))-V(\mathcal{B}_{k-1}(G))$. 

The name $k$-Bell colour graph of the graph $G$ comes from the $k$-Bell number of $G$, which is  the number of partitions of $V(G)$ into at most $k$ independent sets.
Similarly, the $k$-Stirling number of $G$ is the number of partitions of $V(G)$ into exactly $k$ independent sets.  
Notice that the  Bell number,  $B(n,k)$, is the number of partitions of the vertices of $\overline{K}_n$, into at most $k$ independent sets, and the Stirling number of the second kind, $S(n, k)$,  is the number of partitions of the vertices of $\overline{K}_n$ into exactly $k$ independent sets.
Refer to \cite{DP, GT} for pointers to the literature on Bell and Stirling numbers of graphs.

The $k$-Bell colour graph has been considered by Haas under the name the {\em isomorphic colour graph} \cite{Haas}.  This name arises from defining two vertices of $C_k(G)$ to be {\em equivalent} (or, saying that these colourings are isomorphic) if they give rise to the same partition of the vertex set into independent sets.  
The vertices of $\mathcal{B}_k(G)$ correspond to these equivalence classes, with  $[x]$ and $[y]$ being adjacent if and only if some member of $[x]$ is adjacent to a member of $[y]$ in $\mathcal{C}_k(G)$.
The graph $\mathrm{Can}_k^\pi(G)$ is the subgraph of $C_k(G)$ induced by the set of lexicographically least representatives with respect to $\pi$ from the equivalence classes. 
It follows that, for any vertex ordering $\pi$, the graph $\mathrm{Can}_k^\pi(G)$ is a spanning subgraph of $\mathcal{B}_k(G)$.  
Hence, if there exists a vertex ordering $\pi$ such that $\mathrm{Can}_k^\pi(G)$ is connected, or Hamiltonian, then the same is true of $\mathcal{B}_k(G)$.

In Section~2  we give some of the basic properties of  $k$-Bell colour graphs and in Section~3 we show that, for every graph which is neither a complete graph nor a complete graph less an edge, there is a threshold $k_0$, so that  for every $k\ge k_0$, the $k$-Bell colour graph of $G$ is Hamiltonian. In Section 4, we show that for $k\ge 4$, the $k$-Stirling colour graph of any tree with at least $k+1$ vertices is Hamiltonian, but the 3-Stirling colour graph of a star with an odd number of vertices is not Hamiltonian.  
Since we are viewing colourings as partitions of the vertex set into independent sets, through the remaining sections  we will use the terms ``colouring'' and ``partition'' interchangeably.

\section{Known results and basic properties}

In this section we explore properties of the $k$-Bell colour graph of $G$.  While this graph is defined for every non-negative integer $k$, for every $k \geq |V(G)|$ we have $\mathcal{B}_k(G) \cong \mathcal{B}_{|V(G)|}(G)$.  By contrast, the $k$-colour graph of $G$ is  different for every positive integer $k \geq \chi(G)$.  

 Haas proved that, for any integer $k \geq 3$ and any tree $T$ with at least 4 vertices,
 there is an enumeration $\pi$ of  $V(T)$ 
 such that $\mathrm{Can}_k^\pi(T)$ is Hamiltonian.  Hence we have the following.
 
 \begin{theorem}  {\rm \cite{Haas}}
For any tree with at least four vertices, $\mathcal{B}_k(T)$  is Hamiltonian for every $k\ge 3$. 
\label{I3}\label{I3K1n}
\end{theorem}
 
 \begin{corollary}
For any tree with at least three vertices, $\mathcal{S}_3(T)$  has a Hamilton path. 
\label{HPinX3}
\end{corollary}
\begin{proof}
The only tree with 3 vertices is $P_3$, and  $\mathcal{S}_3(P_3) \cong K_1$.
Suppose $T$ is a tree with at least 4 vertices.
By Theorem \ref{I3}, $\mathcal{B}_3(T)$ is Hamiltonian. 
The graph $\mathcal{S}_3(T)$ is obtained by deleting the unique 2-colouring of $T$ from  $\mathcal{B}_3(T)$. 
Deleting a vertex from a graph with a Hamilton cycle gives a graph with a Hamilton path.
\end{proof}

%We note that the canonical map from $\mathcal{C}_k(G)$ and $\mathcal{B}_k(G)$, that is the mapping which takes a colouring in  $V(\mathcal{C}_k(G))$ to the corresponding partition in $V(\mathcal{B}_k(G))$, is not a homomorphism as two adjacent colourings in $\mathcal{C}_k(G)$ could be mapped to the same vertex in $\mathcal{B}_k(G)$.  Still, if one adds a loop on each such vertex in $\mathcal{B}_k(G)$, the resulting graph would be a homomorphic image of $\mathcal{C}_k(G)$.  Hence we get the following proposition. 

%We use {\em $\mathit{col}(G)$} to denote the {\em colouring number of $G$}. 
The {\em colouring number of $G$}, denoted {\em $\mathit{col}(G)$}, is the smallest integer $c$ for which there exists an ordering of the vertices  such that for all $i$, the degree of the $i$th vertex in the subgraph induced by the first $i$ vertices in the ordering is less than $c$. 
%The following theorem is the same as for the $k$-colour graph, and the proof is similar as well.

\begin{proposition} {\rm\cite{Haas}}
If $\mathcal{C}_k(G)$ is connected, then $\mathcal{B}_k(G)$ is connected.
\label{P2.3}
\end{proposition}

%\begin{proof} 
%Suppose $u$ and $v$ in $ V(\mathcal{B}_n(G))$  are two different partitions of $V(G)$.  Assign colours to each of the sets in the two partitions, to get two colourings 
%$u^*$ and $v^*$ $\in V(\mathcal{C}_n(G))$.  As this graph is connected, there is a path between $u^*$ and $v^*$ in $\mathcal{C}_n(G)$. Under the canonical map  from $\mathcal{C}_k(G)$ and $\mathcal{B}_k(G)$, this path is a walk from $u$ to $v$.  Therefore $\mathcal{B}_n(G)$ is connected.
%\end{proof}

Since $\mathcal{C}_k$ is connected when $k \geq \mathrm{col}(G)+1$
\cite{DFFV} (also see \cite{CHJ}), 
the following is now an immediate consequence of Proposition \ref{P2.3}.  

%It can also be proved directly using an argument similar to the one that shows $\mathcal{C}_k$ is connected when $k \geq \mathrm{col}(G)+1$.
%

\begin{corollary}
If $k \geq \mathrm{col}(G)+1$, then $\mathcal{B}_k(G)$ is connected.
\label{I_kConn}
\end{corollary}

The graph $L_{n,n} = K_{n,n}-M$, where $M$ is a perfect matching, shows that this bound is best possible.  
It is easy to see that $col(L_{n,n})=n$.  
On the other hand, the $n$-colouring (partition) where each pair of end vertices of the edges in $M$ are assigned the same colour has no neighbours in $\mathcal{B}_n(L_{n,n})$.

%A family of examples can be generated to show this bound is also best possible for $\mathcal{B}_k(G)$.  Start with $K_{kn,kn}$ and delete $n$ vertex disjoint copies of $K_{k,k}$.
%The colouring number of this graph is 

\begin{proposition}
Let $k$ be a positive integer, and $H$ be a uniquely $k$-colourable graph. If $G$ and $H$ are disjoint, then $\mathcal{B}_k(H \cup G) \cong\mathcal{C}_k(G)$.
\end{proposition}

\begin{proof} 
Suppose that $H$ is a uniquely $k$-colourable graph.
Then $H$ has chromatic number $k$.
Let $\{X_1, X_2, \ldots, X_k\}$ be the unique partition of $V(H)$ into independent sets.

Consider the following mapping $f$ from $V(\mathcal{C}_k(G))$ to $V(\mathcal{B}_k(H \cup G))$.
Let $c$ be a vertex of $\mathcal{C}_k(G)$.
Then $c$ is function from $V(G)$ to $\{1, 2, \ldots, k\}$ such that 
if $xy \in E(G)$ then $c(x) \neq c(y)$. 
Define $f(c)$ to be the partition $\{c^{-1}(i) \cup X_i: 1 \leq i \leq k\}$ of $V(G \cup H)$.
In the mapping $f$, the cells in the unique partition of $V(H)$ 
into $k$ independent sets
distinguish the colour classes  with respect to $c$.
%It is clear that if $c_1c_2 \in E\mathcal{C}_k(G))$ then the corresponding vertices of $\mathcal{B}_k(H \cup G)$ are adjacent.
%Conversely, suppose that 
%$\{p_1^{-1}(i) \cup X_i: 1 \leq i \leq k\}$ and
%$\{p_2^{-1}(i) \cup X_i: 1 \leq i \leq k\}$ are adjacent
%vertices of $\mathcal{B}_k(H \cup G)$.
%Since $H$ is uniquely $k$-colourable, the vertex $x$ such that these partitions are
%identical when restricted to $V(G \cup H) - x$ is a vertex of $G$.
%
%
%
%For any $k$-colouring of $G$, form a partition of $H \cup G$, by combining vertices of the same colour into sets in the partition. Given any partition of $H \cup G$, form a colouring of $G$, by assigning to each vertex, the colour of the vertices of $H$ in its set in the partition.  
This mapping is clearly a isomorphism. 
\end{proof}

Let $G \vee H$ denote the {\em join} of the disjoint graphs $G$ and $H$ and $G \Box H$ denote their {\em Cartesian product}.

\begin{proposition}
Let $G$ and $H$ be disjoint graphs and let $k = |V(G\vee H)|$.  Then  $\mathcal{B}_k(G \vee H) \cong  \mathcal{B}_k(G) \Box \mathcal{B}_k(H)$.
\end{proposition}

\begin{proof}
In the graph $G \vee H$, every vertex in the copy of $G$ is adjacent to every
vertex in the copy of $H$.
Thus, there is a natural 1-1 correspondence between 
the set of ordered pairs $(P_1, P_2)$, where $P_1$ is a partition of $V(G)$ 
into independent sets and 
 $P_2$ is a partition of $V(H)$ into independent sets,  and
 the set of 
partitions of $V(G \vee H)$ into independent sets:
 the ordered pair $(P_1, P_2)$ corresponds to the partition $P_1 \cup P_2$.
 It is clear that 
 $(P_1, P_2)$ and $(Q_1, Q_2)$ are adjacent in 
 $\mathcal{B}_k(G \vee H)$ if and only if 
 $P_1 = Q_1$ and  $P_2Q_2 \in E(\mathcal{B}_k(H))$,
 or 
 $P_2 = Q_2$ and  $P_1Q_1 \in E(\mathcal{B}_k(G))$.
The result follows.
\end{proof}

%\begin{question}
%Is it true that if $\mathcal{B}_k(G)$ is connected, then it is 2-connected?  NO.  Take the 3 colourings of $K_{m,n}$.  It is connected, but the two colouring is a cut vertex.   Same for $k+1$ colouring of a complete $k$-partite graph...ie the $k$ colouring is a cut vertex.
%
%\end{question}

We will require a small variation of the Cartesian product in our subsequent work.  Let $K_{r}\square^+ K_{s}$ be the graph obtained from
$K_{r} \square K_s$ as follows.
Let $(a, b)$ be a vertex of $K_{r} \square K_s$.
Add a new vertex $v$ and edges joining $v$ to $(a, b)$ and
all neighbours of $(a, b)$. 
Clearly any two graphs constructed in this way are isomorphic.
In the lemma below, we will refer to $(a, b)$ as the \emph{cloned vertex}.

  %In addition, we adopt the notation that $xy$ denotes the concatenation of the strings $x$ and $y$.

\begin{lemma} 
Let $r$ and $s$ be positive integers.
% with $rs > 1$.  
 Then, in the graph 
$K_r\square^+ K_s$ there is a Hamilton path from the cloned vertex to every other vertex.
\label{HamPath}
\end{lemma}

\begin{proof}
If $r = 1$, then $K_{r}\square^+ K_{s} \cong K_{s+1}$ and the statement follows.
The case where $s = 1$ is similar. 
For $r > 1$ and $s > 1$, the graph $K_{r}\square K_{s}$ is Hamilton connected unless $r = s = 2$.  Thus, if $r \neq 2$ or $s \neq 2$, the statement follows immediately.  
It is easy to directly check that the statement holds when $r = s = 2$.  
%Let $x$ be one of the duplicated vertices and let $y$ be any other vertex of the graph and suppose $r\le s$.  Unless $r=s=2$,  $K_r\square K_s$ is Hamilton connected and hence the result follows.   If $r=s=2$, then $K_r\square^+ K_s-{x} \cong C_4$ and $x$ is adjacent to three vertices on this cycle.  A Hamilton path from $x$ to $y$ can be formed by travelling from $x$ to a neighbour of $y$ and then around the four cycle to $y$.  If $s\ge 3$, $K_r\square K_s$ is Hamilton connected and hence the result follows.
\end{proof}

\section{Every graph has a Hamiltonian Bell colour graph}

We note that, for any $k \geq n$, the graph $B_k(K_n) \cong K_1$, and hence is not Hamiltonian.
Similarly, for any $k \geq n-1$, the graph $B_k(K_n - e)$ is isomorphic to $K_1$ or $K_2$, and hence is not Hamiltonian.

\begin{theorem}
For any graph $G$ on $n$ vertices which is neither $K_n$ nor $K_n-e$, $\mathcal{B}_{n}(G)$  is Hamiltonian. 
\label{In}
\end{theorem}

\begin{proof}
Let $G$ be a smallest counterexample.  
Let $|V(G)|=n$.
 % We first deal with a special case.   
%Suppose there exist a vertex $v\in V(G)$ of degree $n-1$.  Delete $v$ from $G$ to form $G'$.  Each partition induced by a colouring of $G$ corresponds to a partition induced by the same colouring restricted to $G'$ where $v$ is a member of a singleton.  In particular, $\mathcal{B}_n(G)\cong \mathcal{B}_{n-1}(G')$ and hence is Hamiltonian.  
Since $G$ is not complete, there exist two non-adjacent vertices $x$ and $y$. 
Let $G' = G-x-y$.
%Form $G'$ from $G$ by deleting $x$ and $y$.   

Every colouring of $G$ can be regarded as an extension of a colouring of $G'$.  
Any fixed colouring of $G'$ can be extended to a colouring of $G$ in five possible ways:

(1) both $x$ and $y$ are each assigned the same colour as a vertex of $G'$;

(2) $x$ is assigned a new colour, and $y$ is assigned the same colour as a vertex of $G'$;

(3) $x$ is assigned the same colour as a vertex of $G'$, and $y$ is assigned a new colour;

(4) $x$ and $y$ are assigned different new colours;

(5) $x$ and $y$ are both assigned the same new colour.\\
%
%In any such extension, the vertex $x$ (respectively $y$) can 
%have the same colour as a vertex of $G'$, a new colour, or 
%be in an existing cell of the partition, in a cell with no other vertices, or a cell of cardinality two with $y$ (resp. $x$).    
Let $c$ be a fixed colouring of $G'$.
Suppose that, in any extension of $c$ to a colouring of $G$,
there are $l_1-1$ existing colours available for $x$, and $l_2-1$ existing colours available for $y$.  
Then, since colourings are partitions of $V(G)$ into independent sets,  there are $l_1l_2$ extensions of $c$ corresponding to 
possibilities (1) through (4), and one more corresponding to possibility (5).
The subgraph of $\mathcal{B}_n(G)$ induced by any set of extensions that agree on $x$ is a complete graph.
Similarly,  the subgraph induced by any set of extensions that agree on $y$ is a complete graph.
It follows that the subgraph of $\mathcal{B}_n(G)$ induced by the extensions of $c$ corresponding to possibilities (1) through (4) has 
$K_{l_1}\square K_{l_2}$ as a spanning subgraph.
Since the extension corresponding to possibility (5) has the same neighbours in $\mathcal{B}_n(G)$ as the one corresponding to possibility (4), we have shown that 
the subgraph of $\mathcal{B}_n(G)$ induced by the extensions of $c$ has 
$K_{l_1}\square^+ K_{l_2}$ as a spanning subgraph.

%Let $c$ be a colouring of $G'$.
%
%%%%%
%%Stephen
%%%%
%
%Let $c'$ be an extension of $c$ to $G-x$.  
%There are two possibilities:
%\begin{itemize}
%\item $y$ is assigned the same colour as a vertex of $G'$;
%\item $y$ is assigned a new colour
%\end{itemize}
%
%In the first case, there is some fixed number, say $l_1$ of colourings of $G$ which agree with $c'$ on $G-y$ and hence form a complete graph in   $\mathcal{B}_n(G)$. 
%%%%%
%
%Ignoring the extension where $\{x,y\}$ is a cell, for each fixed extension of $c$ to the vertices of $V(G')\cup\{y\}$, there are some fixed number of selections $l_1\ge 1$ for $x$ to complete the extension to $G$.  The corresponding $l_1$ vertices in $\mathcal{B}_n(G)$ form a clique (any change corresponds to $x$ moving to a different cell).  
%
%Similarly, for each fixed choice of cell for $x$ to extend  $c$, there is some fixed number of selections $l_2\ge 1$ for $y$ to complete the extension and these extensions form a clique in $\mathcal{B}_n(G)$.  Finally, the extension of $c$ where $\{x,y\}$ is a cell is adjacent to each extension of $c$ where $\{x\}$ is a cell and to each extension of $c$ where $\{y\}$ is a cell. It follows that $K_{l_1}\square^+ K_{l_2}$ is a spanning subgraph of the graph induced by the extensions of $c$ to colourings of $G$ in $\mathcal{B}_n(G)$. 

For a colouring $c$ of $G'$, we will denote the extension of $c$ to $G$  in which
$x$ and $y$ are assigned the same new colour by
%$\{x,y\}$ is a cell by using 
$c(xy)$, and    
%For the colouring $c$ of $G'$ we will define 
the extension of $c$ to $G$  
in which $x$ and $y$ are assigned different new colours by
% $\{x\} $ and $\{y\}$ are cells by using 
 $c(x|y)$.  
 Clearly,  if $c_1$ is adjacent to $c_2$ in $\mathcal{B}_n(G')$, then $c_1(xy)$ is adjacent to $c_2(xy)$ and $c_1(x|y)$ is adjacent to  $c_2(x|y)$ in $\mathcal{B}_n(G)$.

Since $G$ is not a complete graph minus an edge, the graph $G'$ has a nonempty vertex set. 

Suppose $G'$ is complete.
Then it has only one colouring, $c$, and every colouring of $G$ is an extension of $c$.  
Therefore there are positive integers $r$ and $s$ such that $\mathcal{B}_n(G)$ is a spanning supergraph of $K_s\square^+ K_r$.  
Since $G$ is not a complete graph minus an edge, 
either $x$ or $y$ must be non-adjacent to at least one vertex of $G'$.
Hence at least one of $r$ and $s$ is at least 2.  
It now follows from  Lemma~\ref{HamPath} that $\mathcal{B}_n(G)$ is Hamiltonian.

Now suppose $G'$ is a complete graph minus an edge. 
Then  $G'$ has exactly two colourings.  
Let $c_1$ be the colouring of $G'$ where the two non-adjacent vertices are assigned the same colour,
and let $c_2$ be the colouring where the two non-adjacent vertices in $G'$ are assigned different colours.  
By Lemma~\ref{HamPath}, there is a Hamilton path $P_1$, between  $c_1(xy)$ and $c_1(x|y)$ in the subgraph of $\mathcal{B}_n(G)$ induced by the extensions of the colouring $c_1$.  
Similarly, there is a Hamilton path between $c_2(x|y)$ and $c_2(xy)$ in the subgraph of $\mathcal{B}_n(G)$ induced by the extensions of the colouring $c_2$. 
The paths $P_1$ and $P_2$  and the edges $c_1(xy)c_2(xy)$ and $c_1(x|y)c_2(x|y)$ are a  Hamilton cycle in $\mathcal{B}_n(G)$.

Finally, suppose $G'$ is neither a complete graph nor a complete graph minus an edge.
Since $G$ is a smallest counterexample, 
$\mathcal{B}_n(G') = \mathcal{B}_{n-2}(G')$ has a Hamilton cycle, $c_1, c_2\ldots c_{m'}$, where $|V(\mathcal{B}_n(G'))|$ $=m'$.
By Lemma~\ref{HamPath}, 
the subgraph of $\mathcal{B}_n(G)$ induced by the extensions of colouring $c_i$ to $G$ 
has a Hamilton path, $P_i$,  from $c_i(xy)$ to $c_i(x|y)$. 

If $m'$ is even then the paths $P_1, P_2, \ldots, P_{m'}$, together with the edges $c_i(xy)c_{i+1}(xy)$ for  even $i<m'$,  the edges $c_i(x|y)c_{i+1}(x|y)$ for odd $i$ and the edge $c_{m'}(xy)c_1(xy)$ are a Hamilton cycle. 

We now consider the situation when $m'$ is odd. 
First, observe that if $deg_G(x)=deg_G(y)=n-2$, then for each $i$, $c_i$ has exactly two extensions, $c_i(xy)$ and $c_i(x|y)$.  
In this case, the edges $c_i(xy)c_{i+1}(xy)$ and $c_i(x|y)c_{i+1}(x|y)$ for $i=1, 2, \ldots m'-1$, in conjunction with $c_1(xy)c_1(x|y)$ and $c_{m'}(xy)c_{m'}(x|y)$ are a Hamilton cycle.  Hence assume that at least one of $x$ and $y$ has degree at most $n-3$.
Without loss of generality, $c_1$ is the colouring in which every vertex of $G'$ is assigned a different colour, so that $n-2$ colours are used and every colour class is a singleton.
Then in 
the colouring $c_2$,
there is exactly one colour class of size 2, say $\{z_1, z_2\}$.   

We claim that $x$ and $y$ are adjacent to every vertex in $V(G') - \{z_1, z_2\}$.  Suppose, by way of contradiction, $x$ is not adjacent to $v \in V(G') - \{z_1, z_2\}$.
Let $w_1$ and $w_2$ be, respectively, be the extensions of $c_1$ and $c_2$ in which $x$ is assigned the same colour as $v$, and $y$ is assigned a new colour.
Note that $w_1$ and $w_2$ are adjacent.
By Lemma~\ref{HamPath}, there is a Hamilton path, $Q_1$, from $c_1(xy)$ to $w_1$ in the subgraph induced by the extensions of $c_1$ to colourings of $G$,  
and  a Hamilton path $Q_2$ from $c_2(x|y)$ to $w_2$ in the subgraph induced by the extensions of $c_2$ to colourings of $G$.  
The paths $Q_1, Q_2, P_3, P_4, \ldots, P_{m'}$, together with the edges $w_1w_2$, $c_{m'}(xy)c_1(xy)$, edges of the form $c_i(xy)c_{i+1}(xy)$ when $i$ is odd ($3\le i< m'$), and of the form  $c_i(x|y)c_{i+1}(x|y)$ when $i$ is even, are a Hamilton cycle in $\mathcal{B}_n(G)$.  A contradiction, as $G$ is a counterexample, which proves the claim.

Next, we claim that $x$ is  adjacent, in $G$, to either $z_1$ or $z_2$ and that $y$ is  adjacent, in $G$, to either $z_1$ or $z_2$ .  Suppose, by way of contradiction, $x$ is  adjacent, in $G$, to neither $z_1$ nor $z_2$.
Let $w_1$ be the extension of $c_1$ 
in which $x$ is assigned the same colour as $z_1$ and $y$ is assigned a new colour.
Let $w_2$ be the extension of $c_2$ in which
$x$, $z_1$ and $z_2$ are all assigned the same colour, and $y$ is assigned a new colour.
The same technique as in the first claim yields a Hamilton cycle in $\mathcal{B}_n(G)$, contradicting $G$ is a counter example, which proves the claim.

Recall $G$ is a minimal counterexample.  By the first claim $x$ is adjacent to every vertex of  $V(G') \setminus \{z_1, z_2\}$.   
By the second claim $x$ is adjacent to either $z_1$ or $z_2$ or both.  
We conclude that $deg_G(x)\ge n-3$ and $x$ is adjacent to every vertex of  $V(G') \setminus \{z_1, z_2\}$.  
The same argument implies $deg_G(y)\ge n-3$ and $y$ is adjacent to every vertex of  $V(G') \setminus \{z_1, z_2\}$.  
Furthermore, as  $x$ and $y$ were chosen to be arbitrary non-adjacent vertices in $G$, it follows that $\delta (G)\ge n-3$.

By our previous argument,  one of $x$ or $y$ has degree at most $n-3$, so that one of them has degree $n-3$, say $x$.
Without loss of generality, $x$ is not adjacent to $z_1$, and is adjacent to every vertex in  $V(G') \setminus \{z_1\}$.

By choice of $c_1$, in the colouring $c_{m'}$ there is exactly one colour class of size two, say $\{a, b\}$, and 
all other colour classes have size 1.
It follows from symmetry and the first claim that  $z_1 \in \{a, b\}$; without loss of generality, $a=z_1$.  
Since $c_2 \neq c_{m'}$, %(for this would imply $\mathcal{B}_n(G')$ has only two vertices, and hence no Hamilton cycle), 
$b \ne z_2$. 
Therefore, $z_1$ is not adjacent to any of the vertices $x$, $z_2$ and $b$, so that
 $deg_G(z_1)\le n-4 < \delta(G)$, a contradiction. It follows that  $G$ does not exist.
 \end{proof}

It follows from Theorem~\ref{In} that, for every other graph $G$, there is a least integer $n_0$ such that $B_k(G)$ is 
Hamiltonian whenever $k \geq n_0$.
We now show that there exist graphs for which the $k$-Bell colour graph is Hamiltonian if and only if $k = |V(G)|$, that is, 
for which $n_0 =  |V(G)|$.

For $t \geq 1$, let $G_t = K_{2t} - M$, where $M = \{x_1y_1, x_2y_2, \ldots, x_ty_t\}$ is a perfect matching.  
Then $\chi(G_t) = t$.  
For each $\ell$ with $0 \leq \ell \leq t$, there is a 1-1 correspondence between the colourings of 
$G_t$ with $t+\ell$ colours and the $\ell$-element subsets of $M$ consisting of the edges of $M$ whose 
ends are assigned different colours.  
%Hence, the graph $S_{t+c}(G_t) \cong \overline{K}_{t \choose c}$.  
%For $c = 1, 2, \ldots, t-1$, this graph is disconnected and hence not Hamiltonian.  
Thus there is a 1-1 correspondence between the vertices of $\mathcal{B}_{t+\ell}(G_t)$ and the binary sequences
of length $t$ with at most $\ell$ ones: the $i$-th term of the sequence equals 1 if $x_i$ and $y_i$ are assigned different colours, 
and equals $0$ otherwise.
Two vertices of $\mathcal{B}_{t+\ell}(G_t)$ are adjacent if and only if the corresponding sequences 
differ in exactly one place.
Thus, $\mathcal{B}_{t+\ell}(G_t)$ is the subgraph of the $t$-dimensional hypercube induced
by the binary sequences with at most $\ell$ ones.
This graph is bipartite, and has bipartition $(A, B)$, where $A$ is the set of binary sequences with an even 
number of ones (and at most $\ell$ ones), and $B$ is 
the set of binary sequences with an odd number of ones (and at most $\ell$ ones).
We claim that if $\ell < t$, then $|A| \neq |B|$, hence $\mathcal{B}_{t+\ell}(G_t)$ is not Hamiltonian.
Now, 
$$|A| - |B|  =  {t \choose 0} - {t \choose 1} + {t \choose 2} - \cdots + (-1)^t{t \choose \ell} = \pm {t-1 \choose \ell} \neq 0,$$
where we have used the result of \cite{Tucker}, page 128, question 44(c).
Therefore, $\mathcal{B}_{t+\ell}(G_t)$ is not Hamiltonian for all $\ell < t$,
that is, for all $\ell + t < 2t = |V(G_t)|$.

\section{The Hamiltonicity of Stirling colour graphs of trees}

In this Section, we explore the Hamiltonicity of the $k$-Stirling colour graph of trees.  The $|V(G)|$-Stirling colour graph is a singleton and not Hamiltonian, therefore we consider first the case where $|V(G)|-1$ colours are allowed.

\begin{lemma}
Let $G$ be a graph on $n$ vertices which is not complete.  Then
\begin{description}
\item{(i)} there is a 1-1 correspondence between the set of colourings of $G$ that use exactly $n-1$ colours and the set of edges of $\overline{G}$;
\item{(ii)} the graph $S_{n-1}(G)$ is isomorphic to the line graph of $\overline{G}$.
\item{(iii)} the graph $S_{n-1}(G)$ is connected if and only if $\overline{G}$ is connected;
\item{(iv)} the graph $S_{n-1}(G)$ has a Hamilton cycle if and only if $\overline{G}$ has a circuit that contains an endpoint of every edge of $\overline{G}$.
\end{description}
\label{LemmaS_{n-1}}
\end{lemma}
\begin{proof}
In a colouring of $G$ that uses exactly $n-1$ colours there are exactly two vertices that are assigned the same colour, say $x$ and $y$.  Then $xy \not\in E(G)$, so $xy \in E(\overline{G})$.  This proves {\em (i)}.  

Let $c_1$ and $c_2$ be two colourings  of $G$ that use exactly $n-1$ colours.  These are adjacent in $S_{n-1}(G)$ if and only if there is a vertex $x$ which belongs to the unique cell of size two in each one.  That is, if and only if $c_1$ and $c_2$ correspond to adjacent edges of $\overline{G}$.  This proves {\em (ii)}.  Statement {\em (iii)} now follows from properties of the line graph, as does statement {\em (iv)}.
\end{proof}

\begin{lemma}
Let $T$ be a tree on $n \geq 5$ vertices.  Then $S_{n-1}(T)$ is Hamiltonian.
\label{LemmaBaseCase}
\end{lemma}
\begin{proof}
The statement follows from  part {\em (iv)} of 
Lemma \ref{LemmaS_{n-1}} on noting that if $T$ is not a star, then $\overline{T}$ is Hamiltonian, and if $T$ is a star, then $\overline{T}$ has a cycle of length $n-1$.
\end{proof}

We now consider the $k$-Stirling colour graph of trees using fewer colours.  We use $Q_n$ to denote the 
%Recall that a hypercube on $n$ vertices, $Q_n$ is Hamiltonian and the vertices of the 
$n$-dimensional hypercube.
% are binary sequences where two vertices are adjacent if they differ in exactly one bit.  
The vertices of $Q_n$ are the binary sequences of length $n$, and
the partition of $V(Q_n)$ into the set of binary sequences with an even number of ones, 
and the set of sequences with an odd number of ones, is a bipartition.
%the string with an even number of ones form one set in the bipartition of $Q_n$.  
We will use the following well known result, which can be proved by induction.

\begin{lemma}\label{HPinCUBE}
Suppose $x, y \in V(Q_n)$ be vertices belonging to different cells of the bipartition.
Then, there is a Hamilton path from $x$ to $y$.
\end{lemma}

Any non-trivial tree, $T$,  is uniquely 2-colourable; hence $\mathcal{B}_2(T)=\mathcal{S}_2(T) \cong K_1$.  The graph $\mathcal{S}_3(T)$ can have a much richer structure.  The star is a special case, similarly to the situation for $k$-colour graphs \cite{CM}.  We now characterize the situations where $\mathcal{S}_3(K_{1,n})$ is, and is not,  Hamiltonian.
%, however, when the two colouring is included (to get $\mathcal{B}_3(K_{1,n})$) the graph is Hamiltonian. 
 By contrast, Theorem~\ref{I3} states that for any tree, the graph $\mathcal{B}_3(T)$ is Hamiltonian.

\begin{theorem}\label{star}
For any $n\ge 2$, $\mathcal{S}_3(K_{1,n})$ is Hamiltonian if and only if $n$ is odd. 
\end{theorem}
\begin{proof}
The vertex of degree $n$ in $K_{1, n}$ will always be in its own cell in a partition of the vertices into independent sets.  Hence we must consider the partitions of $n$ independent vertices (each of which has degree 1) into exactly two cells.  One of the leaves, say $x$, is used to label the two cells and will be considered to correspond to  colour 0.  The colourings can now be listed as the binary sequences of length $n-1$ containing at least one 1.  Two of these are adjacent in $\mathcal{S}_3(K_{1,n})$ if they either differ in exactly one entry 
(corresponding to colourings that agree on 
$K_{1, n} - y$, where $y \neq x$)
%(corresponding to any leaf which is not $x$ changing cells) 
or they differ in every entry 
(corresponding to colourings that agree on
$K_{1, n} - x$).
%(corresponding to $x$ changing cells).  
This graph is obtained from an $n-1$ dimensional hypercube by adding edges joining antipodal vertices,  and deleting the vertex corresponding to the sequence of all zeros.  

For even $n$,  it is easy to verify that the graph is bipartite (the sequences with an even (odd) number of ones are independent).  Hence  $\mathcal{S}_3(K_{1,n})$ is a bipartite graph with an odd number of vertices and not Hamiltonian. For $n=3$,  $\mathcal{S}_3(K_{1,3})=K_3$ and hence is Hamiltonian. 

Let $n\ge 5$ be an odd integer.
By the above discussion, the subgraph of $\mathcal{S}_3(K_{1,n})$ induced by the binary sequences in which the first two elements are 0 is isomorphic to $Q_{n-3} - \{000\ldots0\}$.  As  $Q_{n-3}$ is Hamiltonian, by symmetry, 
there is a Hamilton cycle in $Q_{n-3}$ on which the vertices adjacent to $000\ldots 000$ are $000\ldots 001$ and $000\ldots 010$. It follows that in $Q_{n-3}  - \{000\ldots0\}$
there is a Hamilton path from $000\ldots 001$ to $000\ldots 010$.  This corresponds to a path in $\mathcal{S}_3(K_{1,n})$ containing all of the binary sequences in which the first two elements are 0.    

The subgraph of $\mathcal{S}_3(K_{1,n})$ induced by the binary sequences where the first two elements are either 10 or 11 is  isomorphic to $Q_{n-2}$.  By Lemma~\ref{HPinCUBE} and since $n$ is odd (and therefore the length of each binary sequence of length $n-1$ is even), there is a Hamilton path, $P$, in this graph from $100\ldots 001$ to $111\ldots 101$ which corresponds to a path in $\mathcal{S}_3(K_{1,n})$ containing all of the binary sequences in which the first two elements are 10 or 11.   Further, 
we note that in the subgraph where the first two elements are 10 or 11, each element of the form $11x$ has exactly one neighbour where the first two elements are 10.  It follows that there are two consecutive vertices on $P$ in which the first two entries are 11.

As   $100\ldots 001$ is adjacent to  $000\ldots 001$ and $111\ldots 101$ is adjacent to  $000\ldots 010$, there is a cycle, $C$, containing all of the binary sequences which start with 00, 10 or 11 and,  by the above, 
there are two consecutive vertices on $C$ which begin with 11, say $11x$ and $11y$.

The subgraph of  $\mathcal{S}_3(K_{1,n})$ induced by the binary sequences where the first two elements are 01 is isomorphic to $Q_{n-3}$.  Let $P$ be the  path obtained from $C$ by deleting the edge joining $11x$ and $11y$.
 By Lemma~\ref{HPinCUBE},  there is a path in $\mathcal{S}_3(K_{1,n})$ from $01x$ to $01y$ containing all of the vertices in which the first two elements are 01. It now follows that there is a Hamilton cycle in $\mathcal{S}_3(K_{1,n})$, when $n$ is odd. \end{proof}

From Theorem~\ref{star},  $\mathcal{S}_3(T)$ is not necessarily Hamiltonian. Lemma~\ref{newB3}  is a technical Lemma that can used in the proof that $\mathcal{S}_4(T)$ is Hamiltonian for trees with sufficient vertices.

\begin{lemma}\label{newB3}
For any tree $T$ with at least four vertices, and any vertex $x \in V(T)$, there are distinct vertices $a,b\in V(T)$  and a Hamilton path $P$ in $S_3(T)$ for which the end vertices are $\{A-\{a\}, B-\{a\}, \{a\}\}$ and $\{A-\{b\}, B-\{b\}, \{b\}\}$ where $x\ne a$, $x\ne b$ and $\{A,B\}$ is the 2-colouring of $T$.  
\end{lemma}

\begin{proof}
 Suppose $T$ is the star on $n+1$ vertices.  By Theorem~\ref{I3K1n}, $\mathcal{B}_3(T)$  is Hamiltonian.  The two neighbours of the two colouring of $T$ in every Hamilton cycle of $\mathcal{B}_3(T)$ are of the form $\{\{s\}, X\cup\{b\}, \{a\}\}$ and  $\{\{s\}, X\cup\{a\}, \{b\}\}$ where $s$ is the vertex of degree $n$ and $a$ and $b$ are leaves.  
 %By symmetry, there is a Hamilton cycle with the two neighbours of the two colouring of $T$ in the Hamilton cycle of the form $\{\{s\}, Y\cup\{u\}, \{w\}\}$ and  $\{\{s\}, Y\cup\{w\}, \{u\}\}$ for every choice of leaves $u$ and $w$.  
By the symmetry of stars, this implies the lemma holds for stars.  
 
We need to establish the result for trees which are not stars. The proof is by induction on the number of vertices.   If $|V(T)|=4$, then $T \cong P_4$ and an appropriate Hamilton path is given in Figure~\ref{B3P4} for any choice of the vertex $x$.
  If $|V(T)|=5$, then either $T \cong P_5$  (and appropriate Hamilton paths are given in Figure~\ref{B3P5} for any choice of the vertex $x$) or $T$ is isomorphic to the tree shown in Figure~\ref{fivevertices} (and two appropriate Hamilton paths in $\mathcal{S}_3(T)$ are given depending on the choice of the vertex $x$).  In all three cases the lemma holds.

%%%%P4%%%%%%
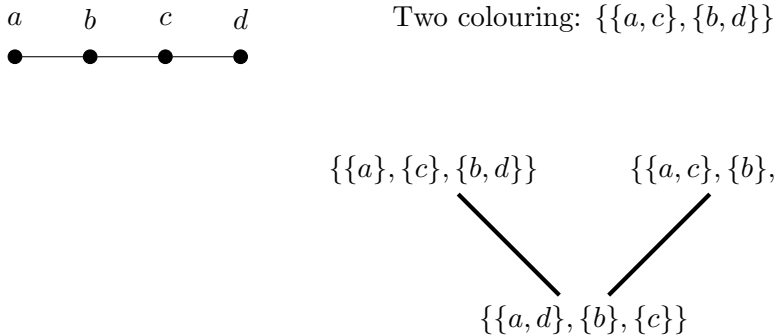
\begin{figure}
\begin{tikzpicture}
	\begin{pgfonlayer}{nodelayer}
		\node [style=vertex] (0) at (-1, 0) {};
		\node [style=vertex] (1) at (0, 0) {};
		\node [style=vertex] (2) at (1, 0) {};
		\node [style=vertex] (3) at (2, 0) {};
		\node [style=textbox] (4) at (-1, .5) {$a$};
		\node [style=textbox] (6) at (0, .5) {$b$};
		\node [style=textbox] (7) at (1, .5) {$c$};
		\node [style=textbox] (8) at (2, .5) {$d$};
	\end{pgfonlayer}
	\begin{pgfonlayer}{edgelayer}
		\draw [style=edge] (0) to (1);
		\draw [style=edge] (1) to (2);
		\draw [style=edge] (2) to (3);
	\end{pgfonlayer}
\end{tikzpicture}
\vspace{-.5in}
 \begin{center}
\begin{tikzpicture}
	\begin{pgfonlayer}{nodelayer}
		\node [style=textbox] (0) at (0, 2) {Two colouring: $\{\{a,c\},\{b,d\}\}$};
		\node [style=textbox] (1) at (-2, 0) {$\{\{a\}, \{c\},\{b,d\}\}$};
		\node [style=textbox] (2) at (0, -2) {$\{\{a,d\},\{b\}, \{c\}\}$};
		\node [style=textbox] (3) at (2, 0) {$\{\{a,c\},\{b\} ,\{d\}\}$};
	\end{pgfonlayer}
	\begin{pgfonlayer}{edgelayer}
%		\draw [style=edge, ultra thick] (1) to (0);
		\draw [style=edge, ultra thick] (1) to (2);
		\draw [style=edge, ultra thick] (2) to (3);
%		\draw [style=edge, ultra thick] (0) to (3);
	\end{pgfonlayer}
\end{tikzpicture}
\end{center}
\caption{The 3-Stirling colour graph of $P_4$}
\label{B3P4} 
\end{figure}

%%%%P5%%%%%%
\begin{figure}
\begin{tikzpicture}
	\begin{pgfonlayer}{nodelayer}
		\node [style=vertex] (-1) at (-2, 0) {};	
		\node [style=vertex] (0) at (-1, 0) {};
		\node [style=vertex] (1) at (0, 0) {};
		\node [style=vertex] (2) at (1, 0) {};
		\node [style=vertex] (3) at (2, 0) {};
		\node [style=textbox] (4) at (-2, .5) {$a$};
		\node [style=textbox] (6) at (-1, .5) {$b$};
		\node [style=textbox] (7) at (0, .5) {$c$};
		\node [style=textbox] (8) at (1, .5) {$d$};
	       \node [style=textbox] (9) at (2, .5) {$e$};
	\end{pgfonlayer}
	\begin{pgfonlayer}{edgelayer}
		\draw [style=edge] (-1) to (0);
		\draw [style=edge] (0) to (1);
		\draw [style=edge] (1) to (2);
		\draw [style=edge] (2) to (3);
	\end{pgfonlayer}
\end{tikzpicture}
\vspace{-.5in}
 \begin{center}
\begin{tikzpicture}
	\begin{pgfonlayer}{nodelayer}
		\node [style=textbox] (0) at (0, 2) {Two colouring: $\{\{a,c,e\},\{b,d\}\}$};
		\node [style=textbox] (1) at (-4, 0.5) {$\{\{a, d\}, \{c\},\{b,e\}\}$};
		\node [style=textbox] (2) at (-4, -1.5) {$\{\{a,d\},\{b\}, \{c, e\}\}$};
		\node [style=textbox] (3) at (-4, -3.5) {$\{\{a,c, e\},\{b\} ,\{d\}\}$};
		\node [style=textbox] (4) at (-4, -5.5) {$\{\{a,c\},\{b,e\} ,\{d\}\}$};
		\node [style=textbox] (5) at (4, 0) {$\{\{a\},\{b,d\}, \{c,e\}\}$};
		\node [style=textbox] (6) at (4, -2) {$\{\{a, e\},\{b,d\}, \{c\}\}$};
		\node [style=textbox] (7) at (4, -4) {$\{\{a, c\},\{b,d\}, \{e\}\}$};
	\end{pgfonlayer}
	\begin{pgfonlayer}{edgelayer}
%		\draw [style=edge, ultra thick] (1) to (0);
		\draw [style=edge, ultra thick] (1) to (2);
		\draw [style=edge,  ultra thick] (2) to (3);
%		\draw [style=edge, bend left, thin] (0) to (3);
		\draw [style=edge, bend left=75, looseness=3.50, thin] (5) to (7);
		\draw [style=edge, bend right=60, looseness=2.50, ultra thick] (1) to (4);
		\draw [style=edge, ultra thick] (5) to (6);
		\draw [style=edge, ultra thick] (6) to (7);
%		\draw [style=edge,ultra thick] (0) to (5);
%		\draw [style=edge, bend right, thin] (0) to (6);
%		\draw [style=edge, bend right, thin] (0) to (7);
		\draw [style=edge, ultra thick] (7) to (4);
		\draw [style=edge] (5) to (2);
		\draw [style=edge,thin] (3) to (4);
	\end{pgfonlayer}
\end{tikzpicture}

\vspace{2cm}

\begin{tikzpicture}
	\begin{pgfonlayer}{nodelayer}
%		\node [style=textbox] (0) at (0, 2) {$\{\{a,c,e\},\{b,d\}\}$};
		\node [style=textbox] (1) at (-4, 0.5) {$\{\{a, d\}, \{c\},\{b,e\}\}$};
		\node [style=textbox] (2) at (-4, -1.5) {$\{\{a,d\},\{b\}, \{c, e\}\}$};
		\node [style=textbox] (3) at (-4, -3.5) {$\{\{a,c, e\},\{b\} ,\{d\}\}$};
		\node [style=textbox] (4) at (-4, -5.5) {$\{\{a,c\},\{b,e\} ,\{d\}\}$};
		\node [style=textbox] (5) at (4, 0) {$\{\{a\},\{b,d\}, \{c,e\}\}$};
		\node [style=textbox] (6) at (4, -2) {$\{\{a, e\},\{b,d\}, \{c\}\}$};
		\node [style=textbox] (7) at (4, -4) {$\{\{a, c\},\{b,d\}, \{e\}\}$};
	\end{pgfonlayer}
	\begin{pgfonlayer}{edgelayer}
%		\draw [style=edge, thin] (1) to (0);
		\draw [style=edge, ultra thick] (1) to (2);
		\draw [style=edge, thin] (2) to (3);
%		\draw [style=edge, bend left, ultra thick] (0) to (3);
		\draw [style=edge, bend left=75, looseness=3.50, thin] (5) to (7);
		\draw [style=edge, bend right=60, looseness=2.50, ultra thick] (1) to (4);
		\draw [style=edge, ultra thick] (5) to (6);
		\draw [style=edge, ultra thick] (6) to (7);
%		\draw [style=edge, thin] (0) to (5);
%		\draw [style=edge, bend right, thin] (0) to (6);
%		\draw [style=edge, bend right, ultra thick] (0) to (7);
		\draw [style=edge, thin] (7) to (4);
		\draw [style=edge, ultra thick] (5) to (2);
		\draw [style=edge, ultra thick] (3) to (4);
	\end{pgfonlayer}
\end{tikzpicture}
\end{center}
\caption{The 3-Stirling colour graph of $P_5$ with two of its Hamilton paths highlighted}
\label{B3P5} 
\end{figure}

 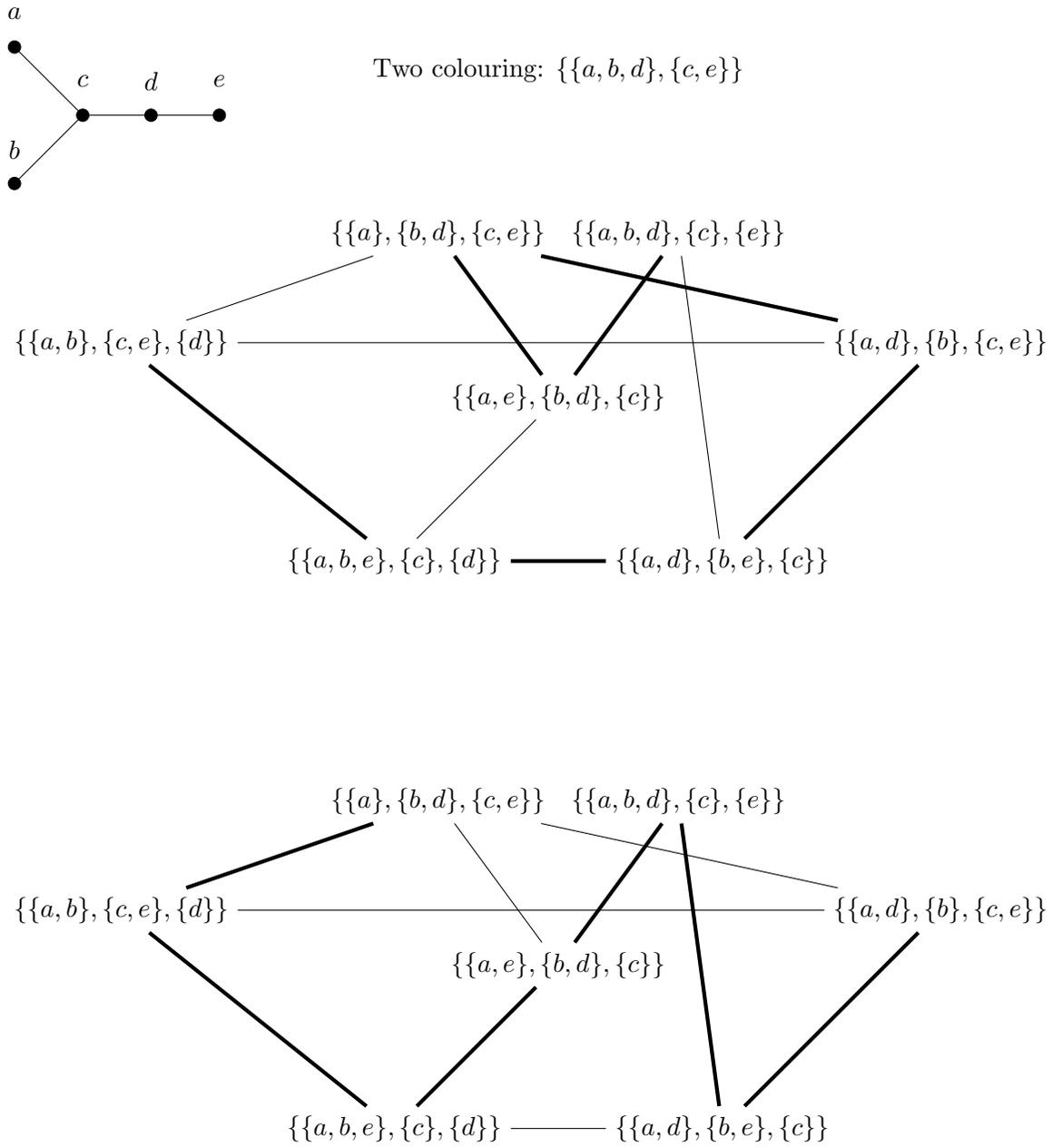
\begin{figure}
  \begin{tikzpicture}
	\begin{pgfonlayer}{nodelayer}
		\node [style=vertex] (0) at (-3, 0) {};
		\node [style=vertex] (1) at (-3, 2) {};
		\node [style=vertex] (2) at (-2, 1) {};
		\node [style=vertex] (3) at (-1, 1) {};
		\node [style=vertex] (4) at (0, 1) {};
		\node [style=textbox] (5) at (0, 1.5) {$e$};
		\node [style=textbox] (6) at (-1, 1.5) {$d$};
		\node [style=textbox] (7) at (-2, 1.5) {$c$};
		\node [style=textbox] (8) at (-3, 2.5) {$a$};
		\node [style=textbox] (9) at (-3, 0.5) {$b$};
	\end{pgfonlayer}
	\begin{pgfonlayer}{edgelayer}
		\draw [style=edge] (1) to (2);
		\draw [style=edge] (0) to (2);
		\draw [style=edge] (2) to (3);
		\draw [style=edge] (3) to (4);
	\end{pgfonlayer}
\end{tikzpicture}

\vspace{-1in}
 \begin{center}
\begin{tikzpicture}[scale=.8]
	\begin{pgfonlayer}{nodelayer}
		\node [style=textbox] (0) at (7, 9) {Two colouring: $\left\{\{a,b,d\},\{c,e\}\right\}$};
		\node [style=textbox] (1) at (-1, 4) {$\left\{\{a,b\},\{c,e\}, \{d\} \right\}$};
		\node [style=textbox] (2) at (4.8, 6) {$\left\{\{a\},\{b, d\}, \{c, e\} \right\}$};
		\node [style=textbox] (3) at (9.2, 6) {$\left\{\{a, b, d\},\{c\}, \{e\} \right\}$};
		\node [style=textbox] (4) at (14, 4) {$\left\{\{a, d\},\{b\}, \{c, e\} \right\}$};
		\node [style=textbox] (5) at (10, 0) {$\left\{\{a, d\},\{b, e\}, \{c\} \right\}$};
		\node [style=textbox] (6) at (7, 3) {$\left\{\{a, e\},\{b, d\}, \{c\} \right\}$};
		\node [style=textbox] (7) at (4, 0) {$\left\{\{a, b, e\},\{c\}, \{d\} \right\}$};
	\end{pgfonlayer}
	\begin{pgfonlayer}{edgelayer}
%		\draw [style=edge, ultra thick] (0) to (1);
%		\draw [style=edge, thin] (0) to (2);
%		\draw [style=edge,  ultra thick] (0) to (3);
%		\draw [style=edge, thin] (0) to (4);
		\draw [style=edge, thin] (1) to (2);
		\draw [style=edge, thin] (1) to (4);
		\draw [style=edge,  ultra thick] (1) to (7);
		\draw [style=edge,  ultra thick] (2) to (4);
		\draw [style=edge,  ultra thick] (2) to (6);
		\draw [style=edge,  ultra thick] (3) to (6);
		\draw [style=edge, thin] (3) to (5);
		\draw [style=edge, ultra thick] (4) to (5);
		\draw [style=edge,  ultra thick] (7) to (5);
		\draw [style=edge, thin] (7) to (6);
	\end{pgfonlayer}
\end{tikzpicture}
\end{center}

\vspace{2cm}

 \begin{center}
\begin{tikzpicture}[scale=.8]
	\begin{pgfonlayer}{nodelayer}
%		\node [style=textbox] (0) at (7, 9) {$\left\{\{a,b,d\},\{c,e\}\right\}$};
		\node [style=textbox] (1) at (-1, 4) {$\left\{\{a,b\},\{c,e\}, \{d\} \right\}$};
		\node [style=textbox] (2) at (4.8, 6) {$\left\{\{a\},\{b, d\}, \{c, e\} \right\}$};
		\node [style=textbox] (3) at (9.2, 6) {$\left\{\{a, b, d\},\{c\}, \{e\} \right\}$};
		\node [style=textbox] (4) at (14, 4) {$\left\{\{a, d\},\{b\}, \{c, e\} \right\}$};
		\node [style=textbox] (5) at (10, 0) {$\left\{\{a, d\},\{b, e\}, \{c\} \right\}$};
		\node [style=textbox] (6) at (7, 3) {$\left\{\{a, e\},\{b, d\}, \{c\} \right\}$};
		\node [style=textbox] (7) at (4, 0) {$\left\{\{a, b, e\},\{c\}, \{d\} \right\}$};
	\end{pgfonlayer}
	\begin{pgfonlayer}{edgelayer}
%		\draw [style=edge] (0) to (1);
%		\draw [style=edge, ultra thick] (0) to (2);
%		\draw [style=edge] (0) to (3);
%		\draw [style=edge, ultra thick] (0) to (4);
		\draw [style=edge, ultra thick] (1) to (2);
		\draw [style=edge] (1) to (4);
		\draw [style=edge, ultra thick] (1) to (7);
		\draw [style=edge] (2) to (4);
		\draw [style=edge] (2) to (6);
		\draw [style=edge, ultra thick] (3) to (6);
		\draw [style=edge, ultra thick] (3) to (5);
		\draw [style=edge, ultra thick] (4) to (5);
		\draw [style=edge] (7) to (5);
		\draw [style=edge, ultra thick] (7) to (6);
	\end{pgfonlayer}
\end{tikzpicture}
\end{center}
\caption{A tree on five vertices and two Hamilton paths in the 3-Srtirling colour graph of the tree.}
\label{fivevertices} 
\end{figure}
  
Suppose the result is true for any tree with at least 4 vertices, but fewer than $n\ge 6$ vertices. Let $T$ be a tree with $n$ vertices which is not a star, and let $x$ be an arbitrary vertex of $T$.  If $x$ is a leaf of $T$, let $x=l_1$ and $l_2$ be leaves of $T$ with corresponding distinct  neighbours $p_1$ and $p_2$ (that is  $p_1\ne p_2$). Set $z=p_1$.  Otherwise, it is possible select leaves $l_1$ and $l_2$ of $T$ so that $x\notin\{l_1, l_2\}$ and $l_1$ and $l_2$ have corresponding distinct neighbours $p_1$ and $p_2$.  Set $z=x$.  Let $T'=T-\{l_1, l_2\}$.

By the induction hypothesis, $\mathcal{S}_3(T')$ contains a Hamilton path, $P$, for which the end vertices are $\{A'-\{y_1\}, B'-\{y_1\}, \{y_1\}\}$ and $\{A'- \{y_2\}, B'- \{y_2\}, \{y_2\}\}$, where $y_1\ne z$,  $y_2\ne z$ and the two colouring of $T'$ is $\{A',B'\}$. 

Every colouring of $T$ is an extension of exactly one colouring of $T'$.  A given 3-colouring of $T'$ may be extended by choosing to add $l_1$ to exactly one of the two cells of the colouring not containing $p_1$, and add $l_2$ to exactly one of the two cells of the colouring not containing $p_2$.  We note that the four corresponding colourings  (or partitions of $V(T)$ into independent sets) induce a 4-cycle in $\mathcal{S}_3(T)$ in which each edge 
arises from 
colourings that agree on $T - l_i$
%corresponds to $l_i$ changing colours 
for some $i$. Note that by similar logic, the extensions of the 2-colouring of $T'$ using two or three colours also form a 4-cycle in $\mathcal{S}_3(T)$ in which each edge arises from 
colourings that agree on $T - l_i$
%$l_i$ changing cells  
for some $i$.  For a colouring $u$ of $T'$, we will call the 4-cycle corresponding to the extensions $u$, $C_u$.

%loss of generality $p_1$ is colour 0 and one of it's non-leaf neighbours is colour 1. 
%Hence a partition may be extended by choosing to colour each of $l_1$ and $l_2$ one of the two colours it is not adjacent to.  We will call the extension of $v_i$  where $l_1$ is coloured $a$ and $l_2$ is coloured $b$, $v_i^{ab}$.  For each $i$, these extensions are of the form $v_i^{1a}$, $v_i^{1b}$, $v_i^{2b}$, $v_i^{2a}$ and these vertices induce a four cycle (in the order listed) in $ \mathcal{B}_3(T)$. 

For colourings $u$ and $v$ of $T'$, let $uv$ be an edge in the Hamilton path in $\mathcal{S}_3(T')$ corresponding to colourings that agree on $T' - y$ for some $y$.
%some vertex $y$ changing cells. 
Now if $y\notin \{p_1,p_2\}$, then every partition in the four extensions of $u$ agrees with one of the four extensions of $v$ on every vertex except $y$.  Hence there is a (perfect) matching between the vertices of $C_u$ and the vertices of $C_v$. Suppose $y=p_1$.  That is, $u$ is a partition $\{\{p_1\}\cup S_1, S_2, S_3\}$
for some independent sets $S_1, S_2$ and $S_3$, and $v$ is the partition $\{S_1,\{p_1\}\cup S_2, S_3\}$. Then the corresponding extensions of $u$ and $v$ where $l_1$ is assigned the same colour as vertices in $S_3$ (or its own colour if $S_3=\emptyset$) are adjacent.  
In particular, the subgraph induced by the four extensions of $u$ and $v$ where $l_1$ is assigned the same colour as vertices in $S_3$ contains a four cycle.   An identical argument can be made to show if $y=p_2$, the subgraph induced by two of the adjacent extensions of $u$ and two of the extension adjacent of $v$ contains a four cycle.  
Note that, in the extensions of 
a colouring $v$, edges in $C_v$ alternately 
correspond to colourings that agree on $T- l_1$ and  
colourings that agree on $T- l_2$.  Furthermore, any fixed colouring of $T'-p_1$ (respectively $T'-p_2$) has  at most 2 extensions to a colouring of $T'$.  Hence   if $u$, $v$, $w$ are three consecutive vertices on $C$, it can not be the case that both $uv$ and $vw$ correspond to colourings that agree on $T' - p_i$ for some $i \in \{1,2\}$.
%$l_1$ and $l_2$ changing cells 
Therefore, from the above argument, if $u$, $v$, $w$ are three consecutive vertices on $C$, the subgraph of  $\mathcal{B}_3(T)$ induced by the vertices belonging to $C_u$, $C_v$ and $C_w$ is a supergraph of the graph shown in Figure~\ref{Fig4}.

The above paragraph justifies the following statements which we will use without reference through the remainder of the proof.  Suppose $u$, $v$, $w$ are three consecutive vertices on $C$.
\begin{itemize}
\item If the edge $uv$ corresponds to 
colourings that agree on $T' - q$, where 
$q \not\in \{p_1, p_2\}$,
%a change in the cell containing a vertex not in $\{p_1, p_2\}$, 
then there is a (perfect) matching from $C_u$ to $C_v$
\item If the  edge $uv$ corresponds to 
colourings that agree on $T' - q$, where 
$q \in \{p_1, p_2\}$, 
%a change in the cell containing $p_i$ for some $i$, 
then the subgraph induced by two of the adjacent extensions of $u$ and two of the adjacent extensions of $v$ contains a four cycle.
\item If the edge $uv$ corresponds to
colourings that agree on $T' - q_1$, 
%is an edge corresponding to a change in the cell containing $p_i$  
and the edge $vw$ corresponds to
colourings that agree on $T' - q_2$, 
where $\{q_1, q_2\} = \{p_1, p_2\}$,
%is an edge corresponding to a change in the cell containing $p_j$, 
%with $\{i,j\}=\{1,2\}$, 
then the
subgraph induced by the vertices of $C_u$, $C_v$ and $C_w$ is a supergraph of the graph shown in Figure~\ref{Fig4}.
\end{itemize} 

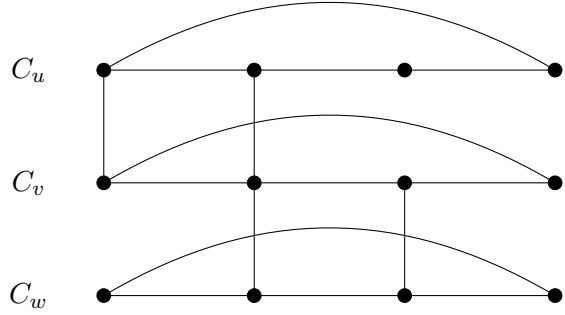
\begin{figure}
\begin{center}
\begin{tikzpicture}
	\begin{pgfonlayer}{nodelayer}
		\node [style=vertex] (0) at (-2, 2) {};
		\node [style=vertex] (1) at (0, 2) {};
		\node [style=vertex] (2) at (2, 2) {};
		\node [style=vertex] (3) at (-2, 0.5) {};
		\node [style=vertex] (4) at (0, 0.5) {};
		\node [style=vertex] (5) at (2, 0.5) {};
		\node [style=vertex] (6) at (-2, -1) {};
		\node [style=vertex] (7) at (0, -1) {};
		\node [style=vertex] (8) at (2, -1) {};
		\node [style=textbox] (9) at (-3, 2) {$C_u$};
		\node [style=textbox] (10) at (-3, -1) {$C_w$};
		\node [style=textbox] (11) at (-3, 0.5) {$C_v$};
		\node [style=vertex] (12) at (4, 2) {};
		\node [style=vertex] (13) at (4, -1) {};
		\node [style=vertex] (14) at (4, 0.5) {};
	\end{pgfonlayer}
	\begin{pgfonlayer}{edgelayer}
		\draw [style=edge] (3) to (4);
		\draw [style=edge] (4) to (5);
		\draw [style=edge] (0) to (1);
		\draw [style=edge] (1) to (2);
		\draw [style=edge] (6) to (7);
		\draw [style=edge] (7) to (8);
		\draw [style=edge] (0) to (3);
		\draw [style=edge] (1) to (4);
		\draw [style=edge] (4) to (7);
		\draw [style=edge] (5) to (8);
		\draw [style=edge] (2) to (12);
		\draw [style=edge, bend left] (0) to (12);
		\draw [style=edge] (5) to (14);
		\draw [style=edge] (8) to (13);
		\draw [style=edge, bend left] (3) to (14);
		\draw [style=edge, bend left] (6) to (13);
	\end{pgfonlayer}
\end{tikzpicture}
\end{center}\label{Fig4}
\caption{A subgraph of the extension of three colourings from a path $uvw$ in $C$ where $uv$ and $vw$ correspond to colourings that agree on 
$T-p_i$ and $T-p_j$ (with $i\ne j$), respectively. }\label{twoinarow}
\end{figure}

\medskip\noindent
The constructive proof used to show that the $3$-colour graph of a tree is Hamiltonian \cite{CM} can now be used.  We provide a proof here for completeness. 

Label the vertices in $C$ in order $v_1, v_2, \ldots , v_k$, where $v_1$ corresponds to the two colouring, say $\{A, B\}$, of $T'$.  

For $i=1, 2, \ldots k$, define $C_i$ sequentially as follows.

\begin{itemize}
\item Define $C_1=C_{v_1}$.

\item Let $s_1$ to be the 2-colouring of $T$.  
We note $s_1$ has a neighbour, say $s_2'$, in $V(C_{v_2})
$. Define the two neighbours of $s_1$ in $C_{v_1}$ as $t'$ and $t''$, where $t'$ corresponds to the extension of $v_1$ where $\{l_1\}$ is a cell and  $t''$ corresponds to the extension of $v_1$ where $\{l_2\}$ is a cell .  If $t'$  has a neighbour  in $V(C_{v_2})$, define $t_1=t'$ and define the neighbour of $t_1$ in $V(C_{v_2})$ as $t_2'$. Otherwise, define $t_1=t''$ and there is a neighbour of $t_1$ in $V(C_{v_2})$ which we define as $t_2'$.

Form $C_2$ as follows. Take $C_1$ and delete the edge $t_1s_1$.  Add the edges $s_1s_2'$ and $t_1t_2'$, together with the path of length three around $C_{v_2}$ which start at $s'_2$ and ends at $t_2'$.

\item For $1<i<k$, suppose $C_{i}$, $s'_i$ and $t'_i$ have been defined in the previous step. Consider the path of length three around $C_{v_i}$ which start at $s'_i$ and ends at $t_i'$.  We note that by this construction, this path is in $C_i$. Find an adjacent pair of vertices on this path with neighbours in $V(C_{v_{i+1}})$ and define them to be $s_i$ and $t_i$ in that order.  (Note that it is possible to have $s_i=s_i'$ or $t_i=t_i'$).  Define the neighbours of $s_i$ and $t_i$ in $V(C_{v_{i+1}})$ to be $s'_{i+1}$ and $t'_{i+1}$, respectively. 

Form $C_{i+1} $ as follows. Take $C_i$ and delete the edge $t_is_i$.  Add the edges $s_is_{i+1}'$ and $t_it_{i+1}'$, together with the path of length three in $C_{v_{i+1}}$ which start at $s'_{i+1}$ and ends at $t_{i+1}'$.

\end{itemize}

By this construction,  $C_k$ is a Hamilton cycle of $\mathcal{B}_3(T)$. It remains to show that $C_k$ satisfies the conclusion of the lemma.  The neighbours of $s_1$, the two colouring of $T$, in $C_k$ are the extension of $v_1$ where $\{l_i\}$ is a cell (for some $i=1,2$) and the extension of $v_2$ where $\{y\}$ is a cell, for some $y\ne z$.   

If $x$ is not a leaf, then $x\ne l_1$, $x\ne l_2$ and, as $x=z$, $x\ne y$. In this case, the lemma follows.  

Suppose $x$ is a leaf.  Then $x=l_1$.  As $y\in V(T')$, $x\ne y$. Furthermore $z=p_1$ implies $y\ne p_1$.  It follows that the extensions of $v_2$ with cell $\{x,y\}$ is a neighbour of $t'$ (the extension of $v_1$ where $\{x\}$ is a cell).  It follows that $t'=t_1$ and therefore   a neighbour of the two colouring of $T$ in $C_k$ is $t''$ (the extension of $v_1$ where $\{l_2\}$ is a cell).  As $l_2\ne x$, this completes the proof. 
\end{proof}

\begin{theorem}
For any tree $T$ with at least $5$ vertices, $\mathcal{S}_4(T)$ is Hamiltonian.
\label{Thm$S_4$}
\end{theorem}
\begin{proof}
The proof is by induction on the number of vertices of $T$.  If $|V(T)|=5$, then Lemma~\ref{LemmaBaseCase} implies $S_4(T)$ is Hamiltonian.
Suppose $T$ is a tree with $n\ge 6$ vertices and suppose further  that for all trees with more than 5 vertices, but fewer than $n$ vertices, the result holds.  

Choose a leaf $l$ of $T$ and define $T'=T-l$.  Let $p$ be the vertex adjacent to $l$ in $T$.

Each  colouring of the vertices of $T$ that uses exactly 4 colours is an extension of exactly one colouring of $T'$ that uses either 3 or 4 colours.   Colourings that extend vertices of $S_3(T')$ correspond to those where $l$ is the only vertex of its colour.  For each colouring in $S_4(T')$, $l$ can be added to exactly three sets of the partition.  The three corresponding colourings of $V(T)$ form a 3-cycle in $\mathcal{S}_4(T)$ in which each edge arises from colourings that agree on $T-l$.
%$l$ changing cells.  
We will call the 3-cycle corresponding to the extensions of the colouring $u$, $C_u$.  

Let $\{A,B\}$ be the 2-colouring of $T'$. By Lemma~\ref{newB3}, there are vertices $a$ and $b$ and a Hamilton path $P$ in $S_3(T')$ for which the end vertices are $\{A-\{a\}, B-\{a\}, \{a\}\}$ and $\{A-\{b\}, B-\{b\}, \{b\}\}$ where $p\ne a$ and $p\ne b$.  It follow that there is a path $P^*$ in $S_4(T)$ which starts at $\{A-\{b\}, B-\{b\}, \{b\}, \{l\}\}$ and ends at  $\{A-\{a\}, B-\{a\}, \{a\}, \{l\}\}$ that contains all of the colourings in which $\{l\}$ is a cell. 

By the induction hypothesis, there is a Hamilton cycle, $C$, in $S_4(T')$.  For each edge $uv$ we consider the edges between $C_u$ and $C_v$  arising from colourings that agree on $T'-y$.
%some vertex $y$ changing sets in the partitions.  
If $y\ne p$ then every colouring among the three extensions of $u$ agrees with exactly one of the extensions of $v$ on every vertex except $y$.  Hence there is a (perfect) matching between the vertices of $C_u$ and the vertices of $C_v$. If $y=p$, then there will be two edges between $C_u$ and $C_v$.  Specifically, if $u$ corresponds to $\{A_1\cup\{p\}, A_2, A_3, A_4\}$ and $v$ corresponds to $\{A_1, A_2\cup\{p\}, A_3, A_4\}$, then $\{A_1\cup\{p\}, A_2, A_3\cup\{l\}, A_4\}$ is adjacent to $\{A_1, A_2\cup\{p\}, A_3\cup\{l\}, A_4\}$ and $\{A_1\cup\{p\}, A_2, A_3, A_4\cup\{l\}\}$ is adjacent to $\{A_1, A_2\cup\{p\}, A_3, A_4\cup\{l\}\}$.  
Note that, if there is a path $uvw$ in $C$ where both $uv$ and $vw$ arise from colourings that agree on $T'-p$,
%$p$ changing sets in the partitions, 
then the endpoints of the matching edges are offset in the sense that the extensions of $u$, $v$ and $w$ have a subgraph of the form shown in Figure~\ref{twoinarow}.

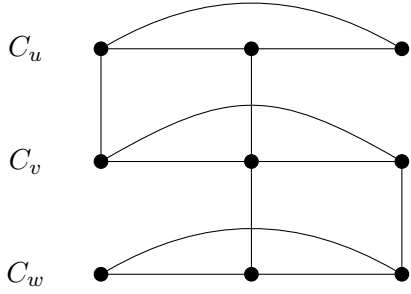
\begin{figure}
\begin{center}
\begin{tikzpicture}
	\begin{pgfonlayer}{nodelayer}
		\node [style=vertex] (0) at (-2, 2) {};
		\node [style=vertex] (1) at (0, 2) {};
		\node [style=vertex] (2) at (2, 2) {};
		\node [style=vertex] (3) at (-2, 0.5) {};
		\node [style=vertex] (4) at (0, 0.5) {};
		\node [style=vertex] (5) at (2, 0.5) {};
		\node [style=vertex] (6) at (-2, -1) {};
		\node [style=vertex] (7) at (0, -1) {};
		\node [style=vertex] (8) at (2, -1) {};
		\node [style=textbox] (9) at (-3, 2) {$C_u$};
		\node [style=textbox] (10) at (-3, -1) {$C_w$};
		\node [style=textbox] (11) at (-3, 0.5) {$C_v$};
	\end{pgfonlayer}
	\begin{pgfonlayer}{edgelayer}
		\draw [style=edge, bend left, looseness=1.25] (3) to (5);
		\draw [style=edge, bend left] (6) to (8);
		\draw [style=edge, bend left] (0) to (2);
		\draw [style=edge] (3) to (4);
		\draw [style=edge] (4) to (5);
		\draw [style=edge] (0) to (1);
		\draw [style=edge] (1) to (2);
		\draw [style=edge] (6) to (7);
		\draw [style=edge] (7) to (8);
		\draw [style=edge] (0) to (3);
		\draw [style=edge] (1) to (4);
		\draw [style=edge] (4) to (7);
		\draw [style=edge] (5) to (8);
	\end{pgfonlayer}
\end{tikzpicture}
\end{center}
\caption{A subgraph of the extension of three colourings from a path $uvw$ in $C$ where both $uv$ and $vw$ correspond to colourings that agree on $T'-p$. }\label{twoinarow}
\end{figure}

Label the vertices of $C$ in order as $v_1v_2\ldots v_k$ with $v_1$ corresponding to the partition $\{A-\{a, b\}, B-\{a,b\}, \{a\}, \{b\}\}$ of the vertices of $T'$.  Without loss of generality assume $p\in B$.  Set $s_1=\{(A-\{a, b\})\cup \{l\}, B-\{a,b\}, \{a\}, \{b\}\}$ (in $C_{v_1}$).  This vertex either has a neighbour in $C_{v_2}$ or $C_{v_k}$.  With out loss of generality, $s_1$ has a neighbour in $C_{v_2}$.  For $i=2, 3,\ldots, k$, choose $s_i$, $t_i$ as described below (where indices are taken modulo $k$).
\begin{itemize}
\item Suppose  $s_{i-1}$ has been previously chosen.  Select $t_{i}$ to be the neighbour of $s_{i-1}$ in $C_{v_{i}}$.
\item Select $s_{i}$ to be a vertex in $C_{v_{i}}$, different from $t_i$ who has a neighbour in $C_{v_{i+1}}$. If possible (under this constraint), select $s_k$ so that $s_k$ is not adjacent to $s_1$.
\end{itemize}
For each $i$, $2\le i \le k$, there is a path, $P_i$, of length 2 containing all the vertices in $C_{v_{i}}$ from $s_i$ to $t_i$.  These paths, together with the edges of the form $s_it_{i+1}$ for $i=1, 2, \ldots, k-1$,  can be used to form a path $P^{**}$ starting at $s_1$ and ending at $s_k$ containing all the vertices belonging to $C_{v_2}, C_{v_3}, \ldots, C_{v_{k-1}}$ and $C_{v_{k}}$.

If $s_k$ is not adjacent to $s_1$, then without loss of generality $s_k$ is adjacent to $\{A-\{a, b\}, B-\{a,b\}, \{a,l\}, \{b\}\}$. The sequence $\{A-\{a, b\}, B-\{a,b\}, \{a,l\}, \{b\}\}$ concatenated with $P^*$ concatenated with $\{A-\{a, b\}, B-\{a,b\}, \{a\}, \{b,l\}\}$ concatenated with $P^{**}$, is a Hamilton cycle of $S_4(T)$.

If $s_k$ is adjacent to $s_1$, then (by the restriction that we must avoid this case if possible) it must be that there are only two edges with one end in $C_{v_k}$ and the other end in $C_{v_1}$ and $t_k$ must be incident with one of these edges.  Furthermore the edge $v_kv_1$ corresponds to 
colourings that agree on $T'-p$.
%$p$ changing sets in the partition. 
By symmetry (as $s_k$ is adjacent to $s_1$), we may assume the edge $v_1v_2$ also corresponds to
colourings that agree on $T'-p$. 
% $p$ changing sets in the partition.  
Hence  $v_1$, $v_2$ and $v_k$ correspond to
colourings that agree on $T'-p$. Observe that there are at most 3 extensions of any fixed colouring of $T'-p$ to a colouring of $T'$. Recalling that $C$ is a Hamilton cycle in $S_4(T')$, it therefore can not also be the case that the edge $v_{k-1}v_k$ corresponds to 
colourings that agree on $T'-p$.
%$p$ changing sets in the partition. 
 Hence there is a matching from $C_{v_{k-1}}$ to $C_{v_k}$.  Set $s'_1=s_1$.  For $2\le i\le k-2$ set $s'_i=s_i$ and $t_i'=t_i$. Choose $s_{k-1}'$ so that $s_{k-1}'$, $s_{k-1}$ and $t_{k-1}$ are the vertices of $C_{v_{k-1}}$.  Choose $t_k'$ to be the neighbour of $s_{k-1}'$ in $C_{v_k}$.  Finally set $s_k'=t_k$. For each $i$, $2\le i \le k$, there is a path, $P_i$, of length 2 containing all the edges from $s_i'$ to $t_i'$.  These paths together with the edges of the form $s_i't_{i+1}'$ (for $i=1, 2, \ldots, k-1$) can be used to form a path $P^{***}$ starting at $s_1'$ and ending at $s_{k}'$ containing all the vertices in $C_{v_2}, C_{v_3}, \ldots, C_{v_{k-1}}$ and $C_{v_{k}}$. Suppose, without loss of generality that $s_{k}'$ is adjacent to  $\{A-\{a, b\}, B-\{a,b\}, \{a,l\}, \{b\}\}$. 
The, the sequence $\{A-\{a, b\}, B-\{a,b\}, \{a,l\}, \{b\}\}$, $P^*$,  $\{A-\{a, b\}, B-\{a,b\}, \{a\}, \{b,l\}\}$, $P^{**}$ is a Hamilton cycle of $S_4(T)$.
%The, the sequence $\{A-\{a, b\}, B-\{a,b\}, \{a,l\}, \{b\}\}$ concatenated with $P^*$ concatenated with $\{A-\{a, b\}, B-\{a,b\}, \{a\}, \{b,l\}\}$ concatenated with $P^{**}$, is a Hamilton cycle of $S_4(T)$.
\end{proof}

In what follows, we make use of a theorem of Choo and MacGillivray \cite{CM} about the existence of Hamilton cycles in graphs with a special structure.  A {\it $C$-graph with vertex  partition
$F_0, F_1, \ldots, F_{N-1}$} is a graph $G$ such that there is a partition $F_0, F_1, \ldots, F_{N-1}$
of $V(G)$ in which, for $i = 0, 1, \ldots, N-1$, the subgraph
induced by $F_i$ is a Hamilton connected graph with at least three vertices.
If $G$ is a $C$-graph with vertex  partition
$F_0, F_1, \ldots, F_{N-1}$, then we define $F_N = F_0$.

If $X$ and $Y$ are disjoint subsets of the vertices of a graph $G$, then the symbol $[X, Y]$ is used to denote the set of edges of $G$ with one end in $X$ and the other end in $Y$.

%%%%ORIGINAL

%\begin{corollary} {\rm \cite{CM}}
%Let $G$ be a $C$-graph with vertex partition $F_0, F_1, \ldots, F_{N-1}$.
%Suppose that, for $j = 0, 1, \ldots, N-1$, the set $[F_j, F_{j+1}]$ contains at least two
%vertex disjoint edges. If there exists $i$, $0 \leq i \leq N-1$, such that $[F_i, F_{i+1}]$ contains at least three
%vertex disjoint edges, then $G$ is Hamiltonian.
%\label{Cor:3edges}
%\end{corollary}

%%%%%NEW VERSION

\begin{corollary} {\rm \cite{CM}}
Let $G$ be a $C$-graph with vertex partition $F_0, F_1, \ldots, F_{N-1}$.
Suppose that, for $j = 0, 1, \ldots, N-1$, the set $[F_j, F_{j+1}]$ contains at least two
vertex disjoint edges. If there exists $i$, $0 \leq i \leq N-1$, such that $[F_i, F_{i+1}]$ contains at least three
vertex disjoint edges, then $G$ is Hamiltonian. Further, the Hamilton cycle is constructed by finding, for $i = 0, 1, \ldots, N-1$, a
Hamilton path through the subgraph induced by $F_0, F_1, \ldots, F_{N-1}$ that begin at an end of a specific edge in $[F_{i-1}, F_i]$ and terminate at an end of a specific edge in $[F_i, F_{i+1}]$. 
\label{Cor:3edges}
\end{corollary}

\begin{lemma}
Let $T$ be a tree on at least $k+2 \geq 7$ vertices and let $l$ be a leaf of $T$.  Suppose that $f_0, f_1, \ldots, f_{N-1}$ is a Hamilton cycle in $S_k(T-l)$.  For $i = 0, 1, \ldots, N-1$, let $F_i$ be the set of extensions of $f_i$ to a $k$-colouring of $T$.  Let $H$ be the subgraph of $S_k(T)$ induced by $F_0 \cup F_1 \cup \cdots \cup F_{N-1}$.
Then
\begin{description}
\item {(i)} $N \geq k-1 \geq 3$;
\item {(ii)} for $i = 0, 1, \ldots, N-1$, the subgraph of $H$ induced by $F_i$ is a complete graph on $k-1$ vertices;
\item {(iii)} for $i = 0, 1, \ldots, N-1$, the set $[F_i, F_{i+1}]$ contains at least three vertex disjoint edges;
\item{(iv) $H$ is a $C$-graph;
\item (v)} For any edge $xy$ in the subgraph of $H$ induced by $F_i$, there is a Hamilton cycle in $H$ that contains the edge $xy$.
\end{description}
\label{LemmaC-graph}
\end{lemma}
\begin{proof}
Let $l'$ be a leaf of $T-l$ and let $T^\prime=T-l-l^\prime$. It follows that $T^\prime$ has at least $k$ vertices and therefore at least one $k$-colouring.  There are $k-1$ extensions of this colouring to a $k$-colouring of $T-l$ and therefore $N=|V(S_k(T-l))|\ge k-1$.  This proves {\em (i)}.

For $1 \leq i \leq N-1$, the colouring $f_i$ has exactly $k-1$ extensions: the vertex $l$ can be inserted into any cell except the one that contains its unique neighbour.  Since any two of these extensions of $f_i$ differ only in the cell containing $l$, they are adjacent in $H$. This proves {\em (ii)}  and {\em (iv)}.

Suppose that $f_i$ and $f_{i+1}$ differ in the cell containing the vertex $x$.  If $x$ is not the unique neighbour of $l$ in $T$, then each colouring $c_i \in F_i$ is adjacent in $H$ to the unique colouring $c_{i+1} \in F_{i+1}$ which is the same when restricted to $V(T) - \{x\}$.  Thus, in this case, $[F_i, F_{i+1}]$ contains $k-1 \geq 4$ vertex disjoint edges.  If $x$ is adjacent to $l$ in $T$ then, since $l$ and $x$ can not belong to the same cell, there are $k-2$ colourings in $c_i \in F_i$ for which there is a unique colouring $c_{i+1} \in F_{i+1}$ which is the same when restricted to $V(T) - \{x\}$ (the remaining colouring $a \in F_i$ also differs from the remaining colouring $b \in F_{i+1}$ in the cell containing $l$).  Thus, in this case, $[F_i, F_{i+1}]$ contains $k-2 \geq 3$ vertex disjoint edges. This proves {\em (iii)}.

Finally, let $xy$ be an edge of $H$ in the subgraph induced by $F_i$.  Let $H^\prime$ be the subgraph of $H$ induced by deleting any edge in $[F_{i-1}, F_i] \cup [F_i, F_{i+1}]$ incident with $x$. The graph $H^\prime$ is a $C$-graph.  Since all the edges which were deleted in the formation of $H'$ were incident with $x$, by  {\em (iii)} for every $0\le j\le N-1$,  the set $[F_j, F_{j+1}]$ contains at least two
vertex disjoint edges.  Furthermore, by {\em (i)}, $N\ge 3$ and therefore by  {\em (iii)} there exists $j$ such that the set $[F_j, F_{j+1}]$ contains at least three
vertex disjoint edges. All the premises in Corollary \ref{Cor:3edges} are therefore satisfied. By  Corollary \ref{Cor:3edges} the Hamilton cycle is constructed by finding, for $i = 0, 1, \ldots, N-1$, a
Hamilton paths through the subgraph induced by $F_0, F_1, \ldots, F_{N-1}$ that begin at an end of a specific edge in $[F_{i-1}, F_i]$ and terminate at an end of a specific edge in $[F_i, F_{i+1}]$.  By construction of $H^\prime$, the vertex $x$ is neither the start nor end of the Hamilton path in the subgraph induced by $F_i$.  Since the subgraph induced by $F_i$ is a complete graph, the edge $xy$ can be included in the Hamilton path  that is used.
\end{proof}

We now outline the main idea in the the proof of the next theorem, and then formalize it in a technical lemma (Lemma \ref{LemmaGluePath}). The argument uses of the same sort of construction as the proof of Theorem \ref{Thm$S_4$}.  Let $k \geq 5$ and let $T$ be a tree on at least $k+2$ vertices.  Let $l$ be a leaf of $T$.  In a $k$-colouring of $T$ either the vertex $l$ belongs to a cell of size at least two, or it belongs to a cell of size one.   In the former case, deleting $l$ gives a $k$-colouring of $T-l$ and, for each $k$-colouring of $T-l$, the vertex $l$ can be inserted into any cell that does not contain its unique neighbour to obtain a $k$-colourings of $T$.  In the latter case, deleting the cell containing $l$ gives a $(k-1)$-colouring of $T-l$, and each $(k-1)$-colouring of $T-l$ can be uniquely extended to a $k$-colouring of $T$ by inserting a cell containing only $l$.
A Hamilton cycle in $S_k(T)$ is constructed from a Hamilton cycle $C$ in $S_{k}(T-l)$ and a Hamilton path in $S_{k-1}(T-l)$ whose ends are joined to consecutive vertices of $C$.

\begin{lemma}
Let $T$ be a tree on at least $k+2 \geq 7$ vertices, let $l$ be a leaf of $T$, and let $s$ be the unique neighbour of $l$ in $T$.  Suppose $S_{k-1}(T-l)$ has a Hamilton path with ends 
$$\alpha = \left\{ X_1 \cup \{w\}, X_2, X_3, \ldots, X_{k-2}, \{y\} \right\},\quad \mathrm{and} \quad
\beta = \left\{ X_1, X_2, X_3, \ldots, X_{k-2}, \{w, y\} \right\},$$
where neither $w$ nor $y$ equals $s$.  
Then, if $S_{k}(T-l)$ has a Hamilton cycle, then so does $S_k(T)$.
\label{LemmaGluePath}
\end{lemma} 
\begin{proof}
Since there is a 1-1 correspondence between the set of $k$-colourings of $T$ in which $l$ belongs to a cell of size one and the  set of $(k-1)$-colourings of $T-l$,  there is a path $P$ in $S_k(T)$ from 
$$\alpha^\prime = \left\{ X_1 \cup \{w\}, X_2, X_3, \ldots, X_{k-2}, \{y\}, \{l\} \right\},\quad \mathrm{to} \quad
\beta^\prime = \left\{ X_1, X_2 , X_3, \ldots, X_{k-1}, \{w, y\}, \{l\} \right\}$$
which contains all of the $k$-colourings where $l$ belongs to a cell of size one.

Let $f_0, f_1, \ldots, f_{N-1}$ be a Hamilton cycle in $S_k(T-l)$ where, without loss of generality, 
$$f_0 =  \left\{ X_1, X_2, X_3, \ldots, X_{k-2}, \{y\}, \{w\} \right\}.$$

For $i = 0, 1, \ldots, N-1$, let $F_i$ be the set of extensions of $f_i$ to a $k$-colouring of $T$.  
Then 
$$a = \left\{ X_1, X_2, X_3, \ldots, X_{k-2}, \{y\}, \{w, l\} \right\},\quad \mathrm{and}\quad
b = \left\{ X_1, X_2, X_3, \ldots, X_{k-2}, \{w\}, \{y, l\} \right\}$$
both belong to $F_0$.
Furthermore both $a \alpha^\prime$ and $b \beta^\prime$ are edges of $S_k(T)$.
The colourings in $F_0 \cup F_1 \cup \cdots \cup F_{N-1}$ are exactly the $k$-colourings of $T$ in which $l$ belongs to a cell of size at least two.
Let $H$ be the subgraph of $S_k(T)$ induced by $F_0 \cup F_1 \cup \cdots \cup F_{N-1}$.
By Lemma \ref{LemmaC-graph}, there is a Hamilton cycle $C$ in $H$ that contains the edge $ab$.

The desired Hamilton cycle in $S_k(T)$ arises from replacing the edge $ab$ of $C$ by the edge $a \alpha^\prime$, followed by the path $P$ and then the edge $\beta^\prime b$.
\end{proof}

\begin{theorem}
Let $r \geq 4$.  If $T$ is a tree on at least $r+1$ vertices, then $S_r(T)$ is Hamiltonian.
\end{theorem}
\begin{proof}
We prove the stronger statement, that if $T$ is a tree on at least $r+1$ vertices and $l$ is a leaf of $T$, then $S_r(T)$ has a Hamiltonian cycle constructed as in Lemma~\ref{LemmaGluePath} or Theorem~\ref{Thm$S_4$} from a Hamilton path in $S_{r-1}(T-l)$ and a
Hamilton cycle in $S_{r}(T-l)$. 

The statement holds for $r=4$ by Theorem \ref{Thm$S_4$}.  
%Hence assume $r \geq 5$.
For the induction hypothesis suppose, for some $k \geq 5$ and all $r$ such that $4 \leq r < k$,
that 
%for any value of $k_0$, with $4\le k_0\le k-1$  
%that 
if $T$ is a tree with at least $r +1$ vertices, then $S_{r}(T)$ has a Hamiltonian cycle constructed from a Hamilton path in $S_{r-1}(T-l)$ and a
Hamilton cycle in $S_{r}(T-l)$ as in Lemma~\ref{LemmaGluePath} or Theorem~\ref{Thm$S_4$}.  
We need to show the statement to be proved holds when $r = k$, and do so
%We show that this statement holds for trees on at least $k+1$ vertices 
by induction on the number of vertices in the tree. 

The basis, when $|V(T)| = k+1$, 
was established in Lemma \ref{LemmaBaseCase}. We note that in this case, for any leaf $l$ of $T$,  $S_{k}(T-l) \cong K_1$. Hence any Hamilton cycle  in $S_k(T)$ is  constructed from the only vertex in $S_{k}(T-l)$ and a Hamilton path in $S_{k-1}(T-l)$.  
Suppose that the statement holds for all trees on between $k+1$ and $n-1$ vertices, and let $T$ be a tree on $n$ vertices.  Note that $n \geq k+2 \geq 7$.

Let $l$ be an end of a longest path in $T$. Then $l$ is a leaf.
Let $s$ be its unique neighbour.   The vertex $s$ either has degree 2, or is adjacent to another leaf $l^\prime \neq l$.  We consider these two possibilities.

\smallskip \noindent
{\bf Case 1}. {\em The vertex $s$ is adjacent to another leaf $l^\prime \neq l$.}

By the induction hypotheses (both), there exists a Hamilton cycle in $S_k(T-l)$ and a Hamilton cycle $C$ in $S_{k-1}(T-l)$.  

The Hamilton cycle $C$ can be assumed to be constructed as in Lemma~\ref{LemmaGluePath} or Theorem~\ref{Thm$S_4$} from a Hamilton path $A$ in $S_{k-2}(T-l-l^\prime)$ and a
Hamilton cycle $B$ in $S_{k-1}(T-l-l^\prime)$.  Let $\alpha\beta$ be an edge of $C$ that joins a vertex in $A$ to a vertex in $B$.  Then, there is a vertex $w \neq s$ such that
$$
\alpha  =  \left\{X_1, X_2, \ldots, X_{k-1}, \{l^\prime\}\right\},\quad \mathrm{and}\quad 
\beta  =  \left\{X_1 - \{w\}, X_2, \ldots, X_{k-1}, \{l^\prime, w\}\right\}.
$$
It follows that $S_{k-1}(T-l)$ has a Hamilton path from $\alpha$ to $\beta$.

Since neither $w$ nor $l^\prime$ equals $s$, the existence of the desired Hamilton cycle in $S_k(T)$ follows from Lemma \ref{LemmaGluePath}.

\smallskip \noindent
{\bf Case 2}. {\em The vertex $s$ has degree two.}

Again, by the induction hypotheses (both), there exists a Hamilton cycle in $S_k(T-l)$ and a Hamilton cycle $C$ in $S_{k-1}(T-l)$. The Hamilton cycle $C$ can be assumed to be constructed as in Lemma \ref{LemmaGluePath} from a Hamilton path $A$ in $S_{k-2}(T-l-s)$ and a Hamilton cycle $B$ in $S_{k-1}(T-l-s)$. 

 Since $k \geq 5$ we have $k-4 \geq 1$. Let $x_1, x_2, \ldots, x_{k-4}$ be vertices of $T-l-s$ such that $T^\prime = T - \left\{l, s, x_1, x_2, \ldots, x_{k-4}\} \right\}$ is a tree. 
Recall that $T$ has at least $k+1 \geq 6$ vertices, so that $T'$ has at least 3 vertices. 
Let  $\{X, Y\}$ be a 2-colouring of $T^\prime$ such that $|Y| \geq 2$. 
Let $$\alpha = \left\{ X, Y, \{x_1\}, \{x_2\}, \ldots, \{x_{k-4}\}, \{s\} \right\}.$$
Note that $\alpha$ is a colouring of $T - l$.
Since $T-l$ and $T'$ are trees, 
at most one vertex of $Y$ can be adjacent to $x_1$.
Hence, 
for some $w \in Y$, the vertex $\alpha$ has a neighbour of the form
$$\beta =  \left\{ X, Y-\{w\}, \{x_1, w\}, \{x_2\}, \ldots, \{x_{k-4}\}, \{s\} \right\}$$
in the collection of colourings that extend those
on  the Hamilton path $A$ to $T - l$. 
Note that $\{s\}$ is a cell in each of these.
It follows that  $\alpha$ and $\beta$  are neighbours on $C$ and therefore, $S_{k-1}(T - l)$ has a Hamilton path from $\alpha$ to $\beta$.

Since  neither $w$ nor $x_1$ equals $s$, the existence of the desired Hamilton cycle in $S_k(T)$ now follows from Lemma \ref{LemmaGluePath}.

\end{proof}

%\section{Future Research??}
%
%
%\begin{question}
%Is it true that for any $k \geq 2+ \mathit{col}(G)$ the graph $\mathcal{B}_k(G)$ is Hamiltonian?
%\end{question}
%
%The same proof as for $\mathcal{C}_k$ shows that there is a Hamilton path.  Closer inspection might reveal that it shows the existence of a Hamilton cycle.

\end{document}